\documentclass[12pt,]{article}
\usepackage{mathrsfs}
\usepackage{amssymb,amsmath}
\usepackage{cases}
\usepackage{amsfonts}
\usepackage{color,xcolor,cite}
\usepackage{bm}
\usepackage[left=2.0cm,right=2.0cm,top=2.0cm,bottom=2.0cm]{geometry}
\usepackage[colorlinks,citecolor=blue,urlcolor=blue]{hyperref}

\numberwithin{equation}{section}
\newenvironment{proof}{{\noindent \it Proof.} }{\hfill $\square$ \par}
\newtheorem{thm}{Theorem}[section]
\newtheorem{prop}{Proposition}[section]

\newtheorem{example}{Example}[section]
\newtheorem{lem}{Lemma}[section]
\newtheorem{definition}{Definition}[section]
\newtheorem{remark}{Remark}[section]

\begin{document}
\title{On distribution dependent stochastic differential equations driven by $G$-Brownian motion}
	\author{De Sun$^{a}$, Jiang-Lun Wu$^{b}$ and Panyu Wu$^{c,}$\footnote{Corresponding author. 
E-mail: wupanyu@sdu.edu.cn} ~\\
	{\small $^a$ School of Mathematics, Shandong University, Jinan 250100, China}\\
{\small $^b$ Department of Mathematics, Computational Foundry, Swansea University, Swansea SA1 8EN, UK}\\
	{\small $^c$ Zhongtai Securities Institute for Financial Studies, Shandong University, Jinan 250100, China}\\
	{\small (Emails: sunde@mail.sdu.edu.cn; j.l.wu@swansea.ac.uk; wupanyu@sdu.edu.cn)}}	
    \date{}
    \maketitle\noindent{}	
    {\bf Abstract:}  Distribution dependent stochastic differential equations have been a very hot subject with extensive studies. On the other hand, under the $G$-expectation framework, stochastic differential equations driven by $G$-Brownian motion (in short form, $G$-SDEs) have received increasing attentions,
and the existence and uniqueness of solutions to $G$-SDEs under Lipschitz and non-Lipschitz conditions have been obtained. Based on these studies,
it is very natural and also important to investigate the $G$-SDEs which are also distribution dependent. In this paper, we are concerned with the well-posedness of the distribution dependent $G$-SDEs. To this end, we first introduce a proper distance of the involved distribution functions and propose a new formulation of the distribution dependent $G$-SDEs. Then, by utilising fix point argument, we establish existence and uniqueness of the solutions of distributed dependent $G$-SDEs under Lipschitz condition. Finally, we derive certain estimates for the solutions of the distribution dependent $G$-SDEs.		

\medskip


\medskip


	\section{Introduction}
We start with a brief account of the history of distribution dependent stochastic differential equations. The prototype model equation was first proposed in 1938 by Vlasov (see the reprint \cite{vlasov}) with the named of  mean field interaction in mathematical physics. Furthermore, inspired by Kac’s profound foundations of kinetic theory \cite{kac}, in the seminal work \cite{mckean}, McKean had formulated the relevant SDEs with coefficients involving distributions of the solutions, hereafter names as McKean-Vlasov equations (also called mean-field SDEs in some literature). It was Sznitman who first established the existence and uniqueness of solutions of McKean-Vlasov SDEs with bounded drift coefficients and constant diffusion coefficients \cite{szn}. Furthermore, with the help of Wasserstein distance, M\'{e}l\'{e}ard \cite{mel} proved the existence and uniqueness of solutions of McKean-Vlasov SDEs  under Lipschitz condition by using the fixed point theorem. Moreover, Buckdahn et al \cite{lipeng} studied the relation of mean-field SDEs  to non-linear PDEs. Moreover, in Wang \cite{wang}, by iterating in distributions, strong solutions of McKean-Vlasov SDEs were constructed. In Huang and Yang \cite{huang}, the existence and uniqueness of the distribution-dependent SDEs with the H\"{o}lder continuous drift coefficients driven by a $\alpha$-stable process were obtained.  Furthermore, R\"{o}ckner and Zhang \cite{roc} studied the situation in which the diffusion coefficients are uniformly non degenerate, bounded, H\"{o}lder continuous and the drift coefficients are integrable, they established the existence and uniqueness of strong solutions as well as weak solutions of McKean-Vlasov SDEs.

    On the other hand, motivated by uncertainty problems, risk measures and super-hedging in finance, Peng \cite{peng1}\cite{peng2} invented a framework of time consistent sublinear expectation, called $G$-expectation. Under the $G$-expectation framework, a new notion of $G$-normal distribution was initiated, which plays the same important role in the theory of sublinear expectation as that of normal distribution in the classical probability theory. Based
on the $G$-normal distribution, a new type of $G$-Brownian motion and the related stochastic calculus of It\^{o}'s type have been developed. Furthermore,
stochastic differential equations driven by $G$-Brownian motion ($G$-SDEs) have been considered by Gao \cite{gao} and Peng \cite{peng2}, wherein the solvability of $G$-SDEs under Lipschitz conditions has been obtained by the contraction mapping theorem. Along this line, Lin \cite{lin2} obtained a pathwise uniqueness result for non-Lipschitz $G$-SDEs with bounded  coefficients. Bai and Lin \cite{bai} further studied the existence and uniqueness of solutions to $G$-SDEs with integral-Lipschitz coefficients. More investigation on $G$-SDEs can also be found in \cite{hurenhe}\cite{hujiliu}\cite{hujiang}\cite{lin1}\cite{xugehu} (and references therein), just mention a few.

      Based on the above studies, it is natural to consider the situation that the $G$-SDEs are also distribution dependent. The preliminary question is then the proper formulation of the distribution dependent $G$-SDEs. It was Sun \cite{sun2020} first considered this topic. Sun introduced the following mean-field $G$-SDE

\begin{equation}\label{sunsde}
{X_t} = x_0 + \int_0^t \hat{\mathbb{E}} [b(s,x,{X_s})]{|_{x = {X_s}}}ds + \int_0^t \hat{\mathbb{E}} [{\sigma _j}(s,x,{X_s})]{|_{x = {X_s}}}d{B_s} + \int_0^t \hat{\mathbb{E}} [{h_{ij}}(s,x,{X_s})]{|_{x = {X_s}}}d{\langle {B^i},{B^j}\rangle _s}
\end{equation}
and showed the existence and uniqueness of solutions under Lipschitz condition, see \cite{sun2020} for details, wherein Sun also introduced the mean-field backward SDE driven by $G$-Brownian motion and established the existence and uniqueness theorem for the equation under Lipschitz condition. Further in Sun \cite{sun2021}, the equation with uniformly continuous coefficients was considered. In the present paper, we want to study the equation (\ref{sunsde})
with distribution dependent coefficients. This motives us to establish a new formulation of distribution dependent $G$-SDEs in $G$-expectation space. To this end, we first construct a new distance for the distribution functions, which turns to be equivalent to the 1-Wasserstein distance in probability space (see our remark \ref{1w} below), then we formulate the well-posedness of these equations, and by utilising fix point argument, we establish the existence and uniqueness of the solutions of our distribution dependent $G$-SDEs with Lipschitz coefficients. We ends our paper by deriving certain estimates for the solutions to distribution dependent $G$-SDEs.		

The rest of our paper is organised as follows. Section 2 presents the necessary preliminaries on sublinear expectation spaces and the $G$-framework. In Section 3, we introduce several complete metric spaces of sublinear functionals. Section 4 is devoted to establishing the existence and uniqueness theorem for distribution dependent $G$-SDEs with Lipschitz coefficients, deriving certain estimates for the solutions and providing an example to support our obtained results.
	
	\section{Notations and preliminaries}
	In this section, we will recall some definitions and results in sublinear expectation spaces and the $G$-framework. The readers may refer to Denis et al \cite{denis},  Gao \cite{gao}, Hu et al \cite{hujisong} and Peng \cite{peng1}, \cite{peng2} for more details.
	Let $\Omega $ be a given nonempty set and $\cal {H}$ be a linear space of real functions defined on $\Omega $ such that if
	${X_1}, \cdots ,{X_n} \in \cal {H}$, then $\varphi \left( {{X_1}, \cdots ,{X_n}} \right) \in \cal {H}$ for each $\varphi  \in {C_{Lip}}({{\mathbb{R}^n}})$, where ${C_{Lip}}({{\mathbb{R}^n}})$ denotes the linear space of functions $\varphi$ satisfying the following Lipschitz condition:
	\begin{center}
		$\left| {\varphi (x) - \varphi (y)} \right| \le {C_\varphi }\left| {x - y} \right|$, for $x,y \in {{\mathbb{R}^n}}, $
	\end{center}
	where  ${C_\varphi }>0$ is the minimal Lipschitz constant for the Lipschtian function$\varphi$.\par
	\begin{definition}\label{sub}
	A functional $\mathbb{E}[\cdot]: \mathcal{H} \rightarrow \mathbb{R}$ is called a sublinear expectation on $\mathcal{H}$, denoted by $\mathbb{E}[\cdot]$, if for any $X, Y \in \mathcal{H}$, the following four conditions are fulfilled  \par
		(1) (Monotonicity) if $X \geq Y$, then $\mathbb{E}[X] \geq \mathbb{E}[Y]$;\par
		(2) (Constants preserving) $\mathbb{E}[c]=c$, for $c \in \mathbb{R}$;\par
		(3) (Sub-additivity) $\mathbb{E}[X+Y] \leq \mathbb{E}[X]+\mathbb{E}[Y]$;\par
		(4) (Positive homogeneity) $\mathbb{E}[\lambda X]=\lambda \mathbb{E}[X]$, for $\lambda \in \mathbb{R}^{+}$.\par
	\end{definition}
	The triple $(\Omega, \mathcal{H}, \mathbb{E})$ is called a sublinear expectation space.\par

	\begin{definition}\label{distribution1}
		Let $n\in\mathbb{N}$ and $X=\left(X_{1}, \ldots, X_{n}\right) $ be a given $n$-dimensional random vector on a sublinear expectation space $\left( {\Omega, \cal {H}, \mathbb{E}} \right)$, the sublinear functional ${\mathbb{F}}_{X}$ defined by
\[
{{\mathbb{F}}_{X}}: \varphi \in {C_{Lip}}\left( {{\mathbb{R}^n}} \right) \mapsto {{\mathbb{F}}_{X}}\left( \varphi  \right):= \mathbb{E}\left[ {\varphi \left( X \right)} \right] \in\mathbb{R}
\]
is called the $n$-dimensional distribution of $X$ under $\left( {\Omega, \cal {H}, \mathbb{E}} \right)$.
	\end{definition}
\begin{remark}
The triple $(\mathbb{R}^n,  {C_{Lip}}(\mathbb{R}^n), {{\mathbb{F}}_{X}})$ forms a sublinear expectation space. In other words, ${{\mathbb{F}}_{X}}$ satisfies monotonicity, constants preserving, sub-additivity and positive homogeneity.
\end{remark}
Next, we introduce the notion of distributions of stochastic processes on the sublinear expectation space $\left( {\Omega, \cal {H}, \mathbb{E}} \right)$.
\begin{definition}\label{distribution2}Let $\left(X_{t}\right)_{t \in[0, T]}$ be an $\mathbb{R}^n$-valued stochastic process on the sublinear expectation space
$\left( {\Omega, \cal {H}, \mathbb{E}} \right)$, the functional process $\left({\mathbb{F}}_{t}^{X}\right)_{t \in[0, T]}$ defined by
\[		
{{\mathbb{F}}^{X}_t}: \varphi \in {C_{Lip}}\left( {{\mathbb{R}^n}} \right) \mapsto {{\mathbb{F}}_t^{X}}\left( \varphi  \right) :={{\mathbb{F}}_{X_t}}\left( \varphi  \right) =\mathbb{E}\left[ {\varphi \left( X_t \right)} \right]\in\mathbb{R}
\]
is called the distribution of $\left(X_{t}\right)_{t \in[0, T]}$ on $\left( {\Omega, \cal {H}, \mathbb{E}} \right)$.
	\end{definition}
	
In the rest of the paper, for arbitrarily fixed $d\in\mathbb{N}$, we let $\Omega  = C_0({{\mathbb{R}}^+;\mathbb{R}^d })$ be the space of all ${\mathbb{R}}^d$-valued continuous paths
${({\omega _t})_{t \in {{\mathbb{R}}^ + }}}$
on ${\mathbb{R}}^+$ with ${\omega _0} = {\bf 0}\in\mathbb{R}^d$, equipped with the following distance
	\[\rho\left(\omega^{1}, \omega^{2}\right):=\sum_{i=1}^{\infty} 2^{-i}\left[\left(\max _{t \in[0, i]}\left|\omega_{t}^{1}-\omega_{t}^{2}\right|\right) \wedge 1\right], \text{ for } \omega^{1}, \omega^{2} \in \Omega.\]
	For each fixed $T \in [0,\infty )$,  set ${\Omega _T}: = \left\{ {{\omega _{ \cdot  \wedge T}}:\omega  \in \Omega } \right\}$ and  let ${B_t}(\omega ) = {\omega _t}$ be the canonical process and
	\[Lip\left( {{{\Omega}_T}} \right): = \left\{ {\varphi \left( {{B_{{t_1}}}, \ldots ,{B_{{t_n}}}} \right):{t_1}, \ldots ,{t_n} \in [0,T], \varphi  \in {C_{b.Lip}}\left( {{\mathbb{R}^{d \times n}}} \right),n \ge 1} \right\},\]
where ${C_{b.Lip}}\left( {{\mathbb{R}^{d \times n}}} \right)$ stands for the linear space of bounded functions in ${C_{Lip}}\left( {{\mathbb{R}^{d \times n}}} \right)$.

Peng \cite{peng1} constructed  a consistent sublinear expectation space $\left( {\Omega ,Lip({\Omega _T}),\hat{\mathbb{E}}} \right)$, called the $G$-expectation space and the canonical process $(B_t)_{t\in[0,t]}$ is called a $G$-Brownian motion. The monotonic and sublinear function $G:\ \mathbb{S}(d)\to \mathbb{R}$ is defined by
$$G(A):=\frac{1}{2}\hat{\mathbb{E}}[\langle AB_1,B_1 \rangle],\ A\in \mathbb{S}(d),$$
where $\langle \, , \,\rangle$ is the scalar product on $\mathbb{R}^d$ and $\mathbb{S}(d)$ denotes the collection of $d \times d$ symmetric matrices.

For each given $p \ge 1$, define $\left\|X\right\|_{L_{G}^p}=(\hat{\mathbb{E}}[|X|^p])^{1/p}$  for $X \in Lip({\Omega _T})$, and denote by $L_G^p( {{\Omega _T}})$ the completion of $Lip({\Omega _T})$
under the norm $\left\|\cdot\right\|_{L_{G}^p}$. Then $\hat{\mathbb{E}}$ can be extended continuously to $L_G^p( {{\Omega _T}})$. 		
Next, recall that a partition ${\pi _T^N}$ of $[0,T]$ is a finite, ordered subset ${\pi _T} = \left\{ {{t_0},{t_1}, \cdots {t_N}} \right\}$ such that $0 = {t_0} < {t_1} <  \cdots  < {t_N} = T$. We further set
		\[\mu \left( {{\pi _T^N
		}} \right) := \max \left\{ {\left| {{t_{i + 1}} - {t_i}} \right|:i = 0,1, \cdots ,N - 1} \right\}.\]
For a given partition ${\pi _T^N}$ of $[0,T]$ and any given ${\xi _k} \in {Lip}\left( {{\Omega _{{t_k}}}} \right), k=0,1,2.\cdots,N-1$,
we define the simple process
		\[{\eta _t} := \sum\limits_{k = 0}^{N - 1} {{\xi _k}} {{\bf{1}}_{\left[ {{t_k},{t_{k + 1}}} \right)}}(t).\]
The totality of all simple processes is denoted by $M_G^{p,0}(0,T)$. Furthermore, we let
$M_G^p(0,T)$ be the completion of $M_G^{p,0}(0,T)$ under the norm
	\[{\left\| \eta  \right\|_{M_G^p(0,T)}}: = {\left( {\hat{\mathbb{E}}\left[ {\int_0^T {{{\left| {{\eta _t}} \right|}^p}} dt} \right]} \right)^{\frac{1}{p}}}.\]
Denote by $\bar M_G^p\left(0, T \right)$ the completion of $ M_G^{p,0}\left(0, T \right)$ under the norm
$$\left\| \eta  \right\|_{\bar{M}_G^p(0,T)}: = {\left( {\int_0^T {\hat{\mathbb{E}}\left[ {{{\left| {{\eta _t}} \right|}^p}} \right]} dt} \right)^{\frac{1}{p}}}.$$
Since $\hat{\mathbb{E}}$ is sub-additive, it is clear that $\bar M_G^p\left(0, T \right) \subset M_G^p\left(0, T \right)$.

For each fixed $\mathbf{a}\in {\mathbb{R}}^d$, let $B_t^{\mathbf{a}} := \left\langle {{\mathbf{a}},{B_t}} \right\rangle $. Then $(B_t^{\mathbf{a}})_{t\ge 0}$ is a $1$-dimensional $G_{\mathbf{a}}$-Brownian motion with ${G_{\mathbf{a}}}\left( \alpha  \right) = \frac{1}{2}\left( {\sigma _{{\mathbf{a}}{{\mathbf{a}}^T}}^2{\alpha ^ + } - \sigma _{ - {\mathbf{a}}{{\mathbf{a}}^T}}^2{\alpha ^ - }} \right)$, where ${\sigma_{\mathbf{a}{\mathbf{a}^T}}^2}=2G(\mathbf{a}\mathbf{a}^T) = {\mathbb{E}}\left[ {{{\left\langle {{\bf{a}},{B_1}} \right\rangle }^2}} \right]$ and
$\sigma_{ - {\mathbf{a}}{{\mathbf{a}}^T}}^2=-2G(-\mathbf{a}\mathbf{a}^T)= - {\mathbb{E}}\left[ { - {{\left\langle {{\bf{a}},{B_1}} \right\rangle }^2}} \right]$.
	\begin{definition}\label{integral1}
		For each $ \eta \in M_{G}^{2,0}(0,T)$ of the form
	$\eta_{t}=\sum_{k=0}^{N-1} \xi_{k} \mathbf{1}_{\left[t_{k}, t_{k+1}\right)}(t),$
		we define the stochastic integral
		\[ I(\eta)=\int_{0}^{T} \eta_{t} d B_{t}^{\mathbf{a}}:=\sum_{k=0}^{N-1} \xi_{k}\left(B_{t_{k+1}}^{\mathbf{a}}-B_{t_{k}}^{\mathbf{a}}\right).\]
		This can be continuously extended to $I:M_{G}^{2}([0, T]) \rightarrow L_{G}^{2}\left(\Omega_{T}\right)$. Then for each $ \eta \in M_{G}^{2}(0,T)$, we define the stochastic integral
		\[\int_{0}^{T} \eta_{t} d B_{t}^{\mathbf{a}}:=I(\eta).\]
	\end{definition}
	The following Lemma was established in \cite{peng2} (see Lemma 3.3.4 therein).
	\begin{lem}\label{inequality1}
		For each $\eta \in M_{G}^{2}(0,T) $,
		\[{\hat{\mathbb{E}}} \left [ \int_{0}^{T}\eta _{t}dB_{t}^{\mathbf{a}} \right ] =0,\]	
		\[\hat{\mathbb{E}} \left [ \left ( \int_{0}^{T}\eta _{t}dB_{t}^{{\mathbf{a}}}    \right )^2 \right ]   \le \sigma _{{\mathbf{a}}{\mathbf{a}}^{T} }^2 \hat{\mathbb{E}} \left [ \int_{0}^{T}\eta _{t}^2dt    \right ] . \]
	\end{lem}
  The quadratic variation process of $B^{\mathbf{a}}$ is defined by
		\[\left\langle B^{\mathbf{a}}\right\rangle_{t}:=\lim _{\mu\left(\pi_T^{N}\right) \rightarrow 0} \sum_{k=0}^{N-1}\left(B_{t_{k+1}^{N}}^{\mathbf{a}}-B_{t_{k}^{N}}^{\mathbf{a}}\right)^{2}=\left(B_{t}^{\mathbf{a}}\right)^{2}-2 \int_{0}^{t} B_{s}^{\mathbf{a}} d B_{s}^{\mathbf{a}},\]
		which is not always a deterministic process as in the classical theory.
	\begin{definition}\label{integral2}
		We specify the mapping  $\mathcal{Q}_{0, T}: M_{G}^{1,0}([0, T]) \rightarrow L_{G}^{1}\left(\Omega_{T}\right)$ via
		\[\mathcal{Q}_{0, T}(\eta)=\int_{0}^{T} \eta_{t} d\left\langle B^{\mathbf{a}}\right\rangle_{t}:=\sum_{k=0}^{N-1}  \xi_{k}\left(\left\langle B^{\mathbf{a}}\right\rangle_{t_{k+1}}-\left\langle B^{\mathbf{a}}\right\rangle_{t_{k}}\right)\]	
	and $\mathcal{Q}_{0, T}$ can be uniquely extended to $ M_{G}^{1}([0, T]) \rightarrow L_{G}^{1}\left(\Omega_{T}\right)$. We still denote this mapping by
		\begin{center}
			$\int_{0}^{T} \eta_{s} d\left\langle B^{\mathbf{a}}\right\rangle_{s}:=\mathcal{Q}_{0, T}(\eta)$, for each  $\eta \in M_{G}^{1}(0, T).$
		\end{center}\par
	\end{definition}
	Let $\mathbf{a}$ and $ \overline{\mathbf{a}}$  be two given vectors in $ \mathbb{R}^{d}$. The mutual variation process of $B^{\mathbf{a}}$  and  $B^{\overline{\mathbf{a}}}$ is defined by
	\[	\left\langle B^{\mathbf{a}}, B^{\overline{\mathbf{a}}}\right\rangle_{t}:=\frac{1}{4}\left[\left\langle B^{\mathbf{a}+\overline{\mathbf{a}}}\right\rangle_{t}-\left\langle B^{\mathbf{a}-\overline{\mathbf{a}}}\right\rangle_{t}\right].\]
	Then, for each  $\eta \in M_{G}^{1}(0, T)$,
	\[\int_{0}^{T} \eta_{t} d\left\langle B^{\mathbf{a}}, B^{\overline{\mathbf{a}}}\right\rangle_{t}:=\frac{1}{4}\left(\int_{0}^{T} \eta_{t} d\left\langle B^{\mathbf{a}+\overline{\mathbf{a}}}\right\rangle_{t}-\int_{0}^{T} \eta_{t} d\left\langle B^{\mathbf{a}-\overline{\mathbf{a}}}\right\rangle_{t}\right).\]\par
The following BDG type inequalities can be found in Gao \cite{gao} (see Theorem 2.1 and 2.2 therein).
\begin{lem}\label{inequality2}
		Let $p \geq 2$, $\eta \in M_{G}^{p}(0, T)$ and $\mathbf{a} \in \mathbb{R}^{d}$. For $0 \le s \le t \le T$,
		\[\begin{aligned}
			\hat{\mathbb{E}}\left[ {\mathop {\sup }\limits_{s \le u \le t} {{\left| {\int_s^u {{\eta _r}} dB_r^{\bf{a}}} \right|}^p}} \right]  &\le {C_p}\sigma _{{\bf{a}}{{\bf{a}}^{\rm{T}}}}^p|t - s{|^{\frac{p}{2} - 1}}\hat{\mathbb{E}}\left[ {\int_s^t {{{\left| {{\eta _u}} \right|}^p}} du} \right] \\& \le {C_p}\sigma _{{\bf{a}}{{\bf{a}}^{\rm{T}}}}^p|t - s{|^{\frac{p}{2} - 1}}\left( {\int_s^t \hat{\mathbb{E}} \left[ {{{\left| {{\eta _u}} \right|}^p}} \right]du} \right),
		\end{aligned}\]
		where $C_{p} > 0$ is a constant independent of $\mathbf{a}, \eta$.	\end{lem}
	\begin{lem}\label{inequality3}
		Let $p \geq 1$, $\eta \in M_{G}^{p}(0, T)$ and ${\bf{a}}, {\bf{\bar a}}\in \mathbb{R}^{d}$ and $0 \le s \le t \le T$, then
		\[\begin{aligned}
			{\hat{\mathbb{E}}}\left[\sup _{s \leq u \leq t}\left|\int_{s}^{u} \eta_{r} d\left\langle B^{\mathbf{a}}, B^{\overline{\mathbf{a}}}\right\rangle_{r}\right|^{p}\right] &\leq
			{\left( {\frac{{\sigma _{({\bf{a}} + {\bf{\bar a}}){{({\bf{a}} + {\bf{\bar a}})}^T}}^2 + \sigma _{({\bf{a}} - {\bf{\bar a}}){{({\bf{a}} - {\bf{\bar a}})}^T}}^2}}{4}} \right)^p}{(t - s)^{p - 1}}\hat{\mathbb{E}}\left[ {\int_s^t {{{\left| {{\eta _u}} \right|}^p}} du} \right]\\&\le
			{\left( {\frac{{\sigma _{({\bf{a}} + {\bf{\bar a}}){{({\bf{a}} + {\bf{\bar a}})}^T}}^2 + \sigma _{({\bf{a}} - {\bf{\bar a}}){{({\bf{a}} - {\bf{\bar a}})}^T}}^2}}{4}} \right)^p}{(t - s)^{p - 1}}\int_s^t \hat{\mathbb{E}} \left[ {{{\left| {{\eta _u}} \right|}^p}} \right]du.
		\end{aligned}\]
	\end{lem}
	\begin{remark}
	We would like to pointed out that Gao \cite{gao} (see Theorem 2.1 and 2.2 therein) proved Lemmas \ref{inequality2} and \ref{inequality3} for $\eta \in \bar M_{G}^{p}(0, T)$. While Lemmas \ref{inequality2} and \ref{inequality3} can be verified, respectively, for $\eta \in  M_{G}^{p}(0, T)$ by utilising similar arguments, therefore we omit their proofs here.
	\end{remark}\par
	In the remaining of this paper, we write by $B^{i}=B^{{\mathbf{e}}_{i} }$  for the $i$-th coordinate of the $G$-Brownian motion $B$, under a given orthonormal basis $\left ( \mathbf{e}_{1},\dots \mathbf{e}_{d}   \right )$ in the space ${\mathbb{R}}^d$. For $1\le i\le j\le d$, we denote
$\left \langle B \right \rangle _{t}^{ij} :={\left \langle B^{i}, B^{j} \right \rangle}_{t}.$

\section{Complete metric spaces}
For any $r >0$, let ${{\cal {D}}_r}$ denote the space of functionals $f$ on ${C_{Lip}}\left( {{\mathbb{R}^n}} \right)$ satisfying monotonicity, constants preserving, sub-additivity and positive homogeneity, and further fulfilling the following
\begin{center}\
$\mathop {\sup }\limits_{{C_\varphi } \le r} \left| {f\left( \varphi  \right) - {{\mathbb{F}}_{\bf{0}}}}(\varphi) \right| = \mathop {\sup }\limits_{{C_\varphi } \le r} \left| {f\left( \varphi  \right) - \varphi \left( {\bf{0}} \right)} \right| < \infty ,$
\end{center}
where ${\mathbb{F}}_{\bf{0}}(\varphi):=\varphi({\bf{0}})$ and ${\bf{0}}$ is the zero vector in $\mathbb{R}^n$.
We define the following metric $d_r$ on ${\cal {D}}_r $
	\begin{center}
		$d_r\left( {{f^1},{f^2}} \right) := \mathop {\sup }\limits_{\scriptstyle{C_\varphi } \le {r}\hfill} \left| {{f^1}\left( \varphi  \right) - {f^2}\left( \varphi  \right)} \right|$, for $f^{1} , f^{2} \in {\cal {D}}_r $.
	\end{center}
	Similarly, let ${\cal {D}}_r^T$ denote the space of functional processes $F=(F_t)_{t\in[0,T]}$ on $ \left[ {0,T} \right]  \times  {C_{Lip}}\left( {{{\mathbb{R}}^n}} \right) $ satisfying that $F_t\in \mathcal{D}_r$ and
	\begin{center}
	$\mathop {\sup }\limits_{t \in [0,T]} \mathop {\sup }\limits_{{C_\varphi } \le r} \left| {F_t\left( \varphi  \right) - {{{\mathbb{F}}_{\bf{0}}}}} \right| = \mathop {\sup }\limits_{t \in [0,T]} \mathop {\sup }\limits_{{C_\varphi } \le r} \left| {F_t\left( \varphi  \right) - \varphi \left( {\bf{0}} \right)} \right| < \infty.$
    \end{center}
Then, we can introduce the following metric $d_r^T$ on ${{\cal {D}}_r^T} $
	\begin{center}
		$d_r^T\left( {{F^1},{F^2}} \right) := \mathop {\sup }\limits_{t \in [0,T]} \mathop {\sup }\limits_{
			\scriptstyle{C_\varphi } \le {r}\hfill} \left| {{F_t^1}(\varphi )}-{{F_t^2}(\varphi )} \right|$, for  $F^{1} , F^{2} \in {\cal {D}}_r^T $.
	\end{center}

	\begin{prop}\label{space}
		The spaces ${{\cal {D}}_r}$ and ${{\cal {D}}^T_r}$ are independent of $r>0$ so that they can be denoted by ${\cal {D}}$ and ${\cal {D}}^T$,
		respectively. Moreover, for any $r>0$, we have 	${d_r} = r{d_1}$ and ${d_r^T} = r{d_1^T}$.
	\end{prop}
\begin{proof}
We only prove the first assertion for ${\cal {D}}$, the proofs of other assertions are similar, so are omitted.
In fact, for any $ f\in \mathcal{D}_1$, we have for arbitrarily fixed $r>0$
\[\mathop {\sup }\limits_{{C_\varphi } \le r} \left| f(\varphi)-{\mathbb{F}}_{\bf{0}}(\varphi) \right|
= r\mathop {\sup }\limits_{{C_{\frac{\varphi }{r}}} \le 1} \left|  f(\frac{\varphi}{r})-{{\mathbb{F}}_{\bf{0}}}(\frac{\varphi}{r})  \right|
\le r\mathop {\sup }\limits_{{C_{\varphi} \le 1}} \left|  f(\varphi)-{\mathbb{F}}_{\bf{0}}(\varphi )  \right| <\infty,\]
thus, $ f\in \mathcal{D}_r$ for $r>0$.

On the other hand, let us arbitrarily fix $r>0$. Then, for any $ f\in \mathcal{D}_r$, we have
\[\mathop {\sup }\limits_{{C_\varphi } \le 1} \left| f(\varphi)-{\mathbb{F}}_{\bf{0}}(\varphi) \right|
= \frac{1}{r}\mathop {\sup }\limits_{{C_{r\varphi}} \le r} \left|  f(r\varphi)-{{\mathbb{F}}_{\bf{0}}}(r\varphi)  \right|
\le  \frac{1}{r}\mathop {\sup }\limits_{{C_{\varphi} \le r}} \left|  f(\varphi)-{\mathbb{F}}_{\bf{0}}(\varphi )  \right| <\infty,\]
thus, $ f\in \mathcal{D}_1$.
Hence, $\mathcal{D}_r\equiv \mathcal{D}_1=:\mathcal{D}$. The proof of  Proposition \ref{space} is completed.
\end{proof}
\begin{remark}\label{2w}
It follows from Proposition \ref{space} that for all $r>0$, $d_r$ are equivalent distances on $\mathcal{D}$, so we will just consider the metric $d_1$ on
$\mathcal{D}$ for simplicity. Similarly, we will only consider the metric $d_1^T$ on the space ${{{\cal {D}}}^T}$.
\end{remark}
\begin{remark}
If $f\in \mathcal{D}$ is the distribution of $X$ on the sublinear expectation space $(\Omega, \mathcal{H},\mathbb{E})$, then it is easy to check that \begin{center}\
$\mathop {\sup }\limits_{{C_\varphi } \le 1} \left| {f\left( \varphi  \right) - {{\mathbb{F}}_{\bf{0}}}}(\varphi) \right| = d_1(\mathbb{F}_X,{\mathbb{F}}_{\bf{0}})=\mathbb{E}[|X|]. $
\end{center}

\end{remark}
\begin{remark}\label{1w}
It is worthwhile to point out that in case the sublinear expectation space $(\Omega,\mathcal{H},\mathbb{E})$ degenerates into a usual single probability space $(\Omega,\mathcal{F},P)$, for $\mu,\nu$ being the corresponding probability measures on $(\mathbb{R}^n ,\mathcal{B}(\mathbb{R}^n))$  induced by the random vectors $X$ and $Y$, respectively, then the 1-Wasserstein distance of $\mu,\nu$ equals to $d_1(\mathbb{F}_X,\mathbb{F}_Y)$. This can be verified by the duality theorem of Kantorovich and Rubinstein \cite{Kan}.
\end{remark}
\begin{prop}\label{complete}
		 Both metric spaces $(\mathcal{D},d_1)$ and $(\mathcal{D}^T,d_1^T)$ are complete metric spaces.
	\end{prop}
	\begin{proof} We will only show that $(\mathcal{D}^T,d_1^T)$ is complete, the other can be proved similarly. The proof will be divided into three steps.\par
	\textbf{Step 1.~}If $\left\{ {{F^n}} \right\}_{n=1}^{\infty}$ is a Cauchy sequence in ${ \cal {D}}^T$, then for any $\varepsilon  > 0$, there exists $N >0$ such that  for all $n,m \ge N$,  $ t \in [0,T]$, $\varphi  \in {C_{Lip}}\left( {{{\mathbb{R}}^n}} \right)$ satisfying ${{C_\varphi }}\le 1$, we have
		\begin{equation}\label{step1}
			\left| {F_t^n(\varphi ) - F_t^m(\varphi )} \right| \le d_1^T({F^n},{F^m}) < \varepsilon.
		\end{equation}
	Therefore, for any fixed $t\in[0,T]$, $\left\{ {F_t^n(\varphi )} \right\}_{n=1}^{\infty}$ is a Cauchy sequence in  $\mathbb{R}$, and it definitely converges. Let ${\tilde F_t}\left( \varphi  \right)$ denote the limit of $\left\{ {F_t^n\left( \varphi  \right)} \right\}_{n=1}^{\infty}$, that is
		\[ {\tilde F_t}\left( \varphi  \right):=\mathop {\lim }\limits_{n \to \infty } F_t^n(\varphi ).\]
		For all $t \in [0,T]$, $\psi  \in {C_{Lip}}\left( {{{\mathbb{R}}^n}} \right)$, we have ${{C_{\frac{\psi}{C_\psi}} }}  \le 1$. Define
		\[{F_t}(\psi): = C_{\psi}{\tilde F_t}\left(\frac{\psi}{C_\psi}\right) = \mathop {\lim }\limits_{n \to \infty } F_t^n(\psi).\]
		Clearly, $F$ is a functional process on ${C_{Lip}}\left( {{{\mathbb{R}}^n}} \right) \times \left[ {0,T} \right]$. \par
		\textbf{Step 2.~} We want to verify that $F \in { \cal {D}}^T$. Taking $m \to \infty$ in inequality (\ref{step1}), we obtain
		for any $\varepsilon  > 0$, there exists $N >0$ such that  for all $n \ge N$, $ t \in [0,T]$, $\varphi  \in {C_{Lip}}\left( {{{\mathbb{R}}^n}} \right)$ satisfying $C_{\varphi}\le 1$, that the following holds
		\begin{equation}\label{step2}
			\left| {F_t^n(\varphi ) - {F_t}(\varphi )} \right| < \varepsilon.
		\end{equation}
Since $F^N \in { \cal {D}}^T$, there exists $ k > 0$ such that for all  $ t \in [0,T]$, $ \varphi  \in {C_{Lip}}\left( {{{\mathbb{R}}^n}} \right)$ satisfying $C_{\varphi}\le 1$,
	    \[\left| {F_t^N\left( \varphi  \right) - \varphi \left( {\bf{0}} \right)} \right| \le k.\]
	    Therefore, for all  $ t \in [0,T]$, $ \varphi  \in {C_{Lip}}\left( {{{\mathbb{R}}^n}} \right)$ satisfying $C_{\varphi}\le 1$, we have
	   \[\left| {{F_t}\left( \varphi  \right) - \varphi \left( {\bf{0}} \right)} \right| \le \left| {F_t^N\left( \varphi  \right) - \varphi \left( {\bf{0}} \right)} \right| + \left| {{F_t}\left( \varphi  \right) - F_t^N\left( \varphi  \right)} \right| \le k + \varepsilon .\]
	    Hence
	   \[\mathop {\sup }\limits_{t \in [0,T]} \mathop {\sup }\limits_{{C_\varphi } \le 1} \left| {{F_t}\left( \varphi  \right) - \varphi \left( {\bf{0}} \right)} \right| < \infty .\]
	    It is straightforward that $F_t$ is monotonic, constants preserved,
	    sub-additive and positive homogeneous, thus $F \in { \cal {D}}^T$.\par
	    \textbf{Step 3.~} From inequality (\ref{step2}), we get
		\[d_1^T\left( {{F^n},{F}} \right) = \mathop {\sup }\limits_{t \in [0,T]} \mathop {\sup }\limits_{
			\scriptstyle{C_\varphi } \le {1}\hfill} \left| {{F_t^n}(\varphi )}-{{F_t}(\varphi )} \right| < \varepsilon .\]
		Therefore, ${ \cal {D}}^T$ is a complete metric space with respect to $d_1^T$. The proof of Proposition \ref{complete} is then completed.
\end{proof}

	\section{Existence and uniqueness of distribution dependent SDEs driven by $G$-Brownian motion}
In this final section, we are concerned with the following distribution dependent SDE driven by $G$-Brownian motion (in short, distribution dependent $G$-SDE)
	\begin{equation}\label{DDGSDEs}
		X_{t}=x_0+\int_{0}^{t} b\left(s, X_{s}, {\mathbb{F}}_{s}^{X}\right) d s+\int_{0}^{t} h_{i j}\left(s, X_{s}, {\mathbb{F}}_{s}^{X}\right) d\langle B\rangle _{s}^{i j}+\int_{0}^{t} \sigma_{j}\left(s, X_{s}, {\mathbb{F}}_{s}^{X}\right) d B_{s}^{j},\ t \in[0, T].
	\end{equation}
Here we are using Einstein summation convention, that is, the repeated indices $i$ and $j$ in the above equation means tacitly the summation over those indices.
The initial data $x_0 \in {\mathbb{R}^n}$ is a constant vector, $\left({\mathbb{F}}_{t}^{X}\right)_{t \in[0, T]}$ is the distribution of $\left(X_{t}\right)_{t \in[0, T]}$, $b,h_{i j},\sigma_{j}: [0, T] \times {\mathbb{R}^n} \times {\cal {D}} \rightarrow {\mathbb{R}^n}$ are functionals satisfying the following assumptions (H1) and (H2):\\
		(H1) there exists $K>0$ such that for all $t \in[0, T],\  x, y \in \mathbb{R}^n, f^1, f^2 \in \cal{D}$, we have
				 \[\begin{aligned}&
		\left| {b(t,x,{f^1}) - b(t,y,{f^2})} \right| + \left| {{h_{ij}}(t,x,{f^1}) - {h_{ij}}(t,y,{f^2})} \right| + \left| {{\sigma _j}(t,x,{f^1}) - {\sigma _j}(t,y,{f^2})} \right| \\&\leq K(|x-y|+d_1(f^1, f^2));
		\end{aligned}\]
		(H2) ${b( \cdot ,x,F_{\cdot}),{h_{ij}}( \cdot ,x,F_{\cdot}),{\sigma _j}( \cdot ,x,F_{\cdot}) \in L^2(0,T;{\mathbb{R}^n})}$,  for each  $x \in {\mathbb{R}^n},F \in {\cal {D}}^T$.

By a solution, it is meant to be a stochastic process $(X_t)_{t\in [0,T ]} \in\bar{M}^2_G(0, T ; \mathbb{R}^n)$ fulfilling the distribution dependent  $G$-SDE (\ref{DDGSDEs}).

To ensure that (\ref{DDGSDEs}) is well-posed, primarily the integrands with respect to $d B_{s}^{j}$ should be in $M_G^2\left(0, T ; {\mathbb{R}^n}\right)$ and  the integrands with respect to $d\langle B\rangle _{s}^{i j}$ or $ds $ should be in $M_G^1\left(0, T ; {\mathbb{R}^n}\right)$.
Thus, we need the following lemma.
\begin{lem}\label{app1}
	Fix any $p\ge 1$, let $\zeta: [0, T] \times {\mathbb{R}^n} \times {\cal {D}} \rightarrow {\mathbb{R}^n} $ be a functional such that $\zeta ( \cdot ,x,{{\mathbb{F}}^x}) \in L^p(0,T;{{\mathbb{R}}^n})$ for each $x \in {{\mathbb{R}}^n}$. If $\zeta$ satisfies the Lipschitz condition in the sense that for all $t \in[0, T]$, $x,y \in {{\mathbb{R}}^n}$, ${f^1},{f^2} \in {\cal D}$,
	\begin{equation}\label{lip}
		\left| {\zeta (t,x,{f^1}) - \zeta (t ,y,{f^2})} \right| \le K\left( {\left| {x - y} \right| + {d_1}\left( {{f^1},{f^2}} \right)} \right),
	\end{equation}
	then, we claim that $\zeta ( \cdot ,{X_ \cdot },{\mathbb{F}}_ \cdot ^X)$ is an element in $M_G^p(0,T;{{\mathbb{R}}^n})$ for any $X_ \cdot  \in  \bar{M}_G^p(0,T;{{\mathbb{R}}^n})$.
\end{lem}
\begin{proof}
The proof is motivated by Lemma 5.1 in Bai and Lin \cite{bai}. For readers convenience, we present our proof here.
	Without loss of generality, we only give the proof for the one dimensional case. For $X_ \cdot  \in  \bar{M}_G^p(0,T)$, choose ${\{ {X^N_ \cdot }\} _{N=1}^{\infty}} \subset M_G^{p,0}(0,T)$ such that ${\left\| {X_ \cdot ^N - {X_ \cdot }} \right\|_{ \bar{M}_G^p(0,T)}} \to 0 $ as $N \to \infty$, where $X^N_\cdot$ has the following simple form
	\[X_t^N = \sum\limits_{k = 0}^{N - 1} {{\xi _k}{I_{[{t_k},{t_{k + 1}})}}(t)}, \quad {\xi _k} \in {Lip}\left( {{\Omega _{{t_k}}}} \right).\]
Then, by the Lipschitz condition (\ref{lip}) and H\"{o}lder inequality of $\hat{\mathbb{E}}$ (Proposition 1.4.2 in Peng \cite{peng2}), we get
    \[\begin{aligned}&
	\hat {\mathbb{E}}\left[ {\int_0^T {{{\left| {\zeta (t,X_t^N,{\mathbb{F}}_t^{{X^N}}) - \zeta (t,{X_t},{\mathbb{F}}_t^X)} \right|}^p}dt} } \right]\\&
	\le 2^{p-1}{K^p}\hat {\mathbb{E}}\left[ {\int_0^T {\left( {{{\left| {X_t^N - {X_t}} \right|}^p} + {{\left[ {{d_1}\left( {{\mathbb{F}}_t^{{X^N}},{\mathbb{F}}_t^X} \right)} \right]}^p}} \right)dt} } \right]\\&
	= 2^{p-1}{K^p}\hat {\mathbb{E}}\left[ {\int_0^T {\left( {{{\left| {X_t^N - {X_t}} \right|}^p} +\mathop {\sup }\limits_{
			\scriptstyle{C_\varphi } \le {1}\hfill}\left|\hat{\mathbb{E}}[\varphi(X_t^{N})]-\hat{\mathbb{E}}[\varphi(X_t)]\right|^p }\right)dt} }\right]\\
&\le {(2K)^p}{\int_0^T \hat{\mathbb{E}}\left[ {{{\left| {X_t^N - {X_t}} \right|}^p}}  \right]dt}  \to 0,\quad {\rm{as}}~N \to \infty .
    \end{aligned}\]
	Hence, it is suffices to prove that $\zeta ( \cdot ,X_ \cdot ^N,\mathbb{F}_ \cdot ^{{X^N}}) \in M_G^p(0,T)$, that is, $\zeta ( \cdot ,{\xi _k},{\mathbb{F}^{{\xi _k}}}){I_{[{t_k},{t_{k + 1}})}}( \cdot ) \in M_G^p(0,T)$ for each $k \in \mathbb{N}$. For simplicity of the notation, we want to make a new assertion which  is equivalent to the one stated above: for fixed $T \ge 1$, if $\eta  \in {Lip}\left( {{\Omega _1}} \right)$, then $\zeta ( \cdot ,\eta ,{\mathbb{F}^\eta }){I_{[1,T)}}( \cdot ) \in M_G^p(0,T)$. Before we proceed further, let us prove the assertion.\par
	Since $\eta  \in {Lip}\left( {{\Omega _1}} \right)$, there exists an $M>0$ such that $\eta  \in \left[ { - M,M} \right]$. For each $n \in \mathbb{N}$, there is an open cover ${\left\{ {{G_i}} \right\}_{i \in I}}$ of $\mathbb{R}$ with the Lebesgue measure $\lambda(G_i)<\frac{1}{n}$ for each $i\in I$. By the partition of unity theorem, there exists a family $\{\phi_i^n\}_{i\in I}$ of $C_0^{\infty} (\mathbb{R})$-valued functions  such that
for each $i \in I$, supp$\left( {\phi _i^n} \right) \in {G_i}$, $0 \le \phi _i^n \le 1$, and for each $x \in \mathbb{R}$, $\sum\limits_{i \in I} {\phi _i^n\left( x \right)}  = 1$. Moreover, there exists a finite sub-family of $\{\phi _i^n\}_{i\in I}$, denoted by $\{\phi _i^n\}_{1\le i\le N(n)}$, such that for $x \in \left[ { - M,M} \right]$, $\sum\limits_{i = 1}^{N(n)} {\phi _i^n\left( x \right)}  = 1$. Choosing, for each $i = 1, \cdots, N(n)$, a point $x_i^n$ such that $\phi _i^n\left( {x_i^n} \right) > 0$.  Then let
	\[{\zeta ^n}\left( {t,x,{{\mathbb{F}}^x}} \right) := \sum\limits_{i = 1}^{N(n)} {\zeta \left( {t,x_i^n,{{\mathbb{F}}^{x_i^n}}} \right)} \phi _i^n\left( x \right).\]
	We have
	\[\begin{aligned}&
		\left|  \zeta \left( {t,\eta ,{{\mathbb{F}}^\eta }} \right){I_{[1,T)}}(t)-{\zeta ^n}\left( {t,\eta ,{{\mathbb{F}}^\eta }} \right){I_{[1,T)}}(t) \right|\\
&\le \sum\limits_{i = 1}^{N(n)} {\left|   {\zeta \left( {t,\eta ,{{\mathbb{F}}^\eta }} \right)}-\zeta \left( {t,x_i^n,{{\mathbb{F}}^{x_i^n}}} \right)\right|}\phi _i^n\left( \eta  \right)\\&
		\le K\sum\limits_{i = 1}^{N(n)} {\left( {\left| {\eta  - x_i^n} \right| + {d_1}\left( {{F^\eta },{F^{x_i^n}}} \right)} \right)} \phi _i^n\left( \eta  \right)\\&
\le K\sum\limits_{i = 1}^{N(n)} {\left( {\left| {\eta  - x_i^n} \right| +  \mathop {\sup }\limits_{
			\scriptstyle{C_\varphi } \le {1}\hfill}\left|\hat{\mathbb{E}}[\varphi(\eta)]-\varphi(x_i^n) \right|}    \right)} \phi _i^n\left( \eta  \right)\\&
		\le K\sum\limits_{i = 1}^{N(n)} {\left( {\left| {\eta  - x_i^n} \right| + \hat {\mathbb{E}}\left[ {\left| {\eta  - x_i^n} \right|} \right]} \right)} \phi _i^n\left( \eta  \right) \\&\le \frac{2K}{n},\quad 1 \le t < T,
	\end{aligned}\]
	which implies that ${\zeta ^n}\left( { \cdot ,\eta ,{F^\eta }} \right){I_{[1,T)}}( \cdot )$ converges to $\zeta \left( { \cdot ,\eta ,{F^\eta }} \right){I_{[1,T)}}( \cdot )$ in $M_G^p\left( {0,T} \right)$. Therefore, it suffices to prove that ${\zeta ^n}\left( { \cdot ,\eta ,{F^\eta }} \right){I_{[1,T)}}( \cdot )$ belongs to $M_G^p\left( {0,T} \right)$, that is, $\zeta \left( { \cdot ,x_i^n,{{\mathbb{F}}^{x_i^n}}} \right)\phi _i^n\left( \eta  \right){I_{[1,T)}}( \cdot ) \in M_G^p\left( {0,T} \right)$, $i = 1, \cdots, N(n)$, which can be deduced  from Lemma 5.2 in Bai and Lin \cite{bai}. The proof of Lemma \ref{app1} is thus completed.
\end{proof}

\begin{remark} When $p = 2$, all the coefficients in the distribution dependent $G$-SDE (\ref{DDGSDEs}) satisfy the conditions of Lemma \ref{app1}
under the assumptions (H1) and (H2). Therefore, the $G$-stochastic integrals in the $G$-SDE (\ref{DDGSDEs}) are well defined for any solution $ \left(X_{t}\right)_{t \in[0, T]} \in  \bar M_G^2\left(0, T ; {\mathbb{R}^n}\right)$.
\end{remark}

Our main result can be then formulated as follows.
	\begin{thm}\label{solution}
		Suppose the assumptions (H1) and (H2) hold. Then there exists a unique solution $ \left(X_{t}\right)_{t \in[0, T]} \in  \bar M_G^2\left(0, T ; {\mathbb{R}^n}\right)$ to the distribution dependent $G$-SDE (\ref{DDGSDEs}).
	\end{thm}
	Our proof of Theorem \ref{solution} is inspired by \cite{mel} and \cite{szn} and is based on a fixed point theorem. One considers the mapping $U:{\cal {D}}^{T}\to{\cal {D}}^{T}$ which associates with $F$ the distribution of $X^{F}$, that is $U(F)={\mathbb{F}}^{{X^F}}$, where $X^F$ fulfils the following
	\begin{equation}\label{GSDEs}
		X_{t}^{F}=x_0+\int_{0}^{t} b\left(s, X_{s}^{F}, F_{s}\right) d s+\int_{0}^{t} h_{i j}\left(s, X_{s}^{F}, F_{s}\right) d\langle B\rangle_{s}^{i j}+\int_{0}^{t} \sigma_{j}\left(s, X_{s}^{F}, F_{s}\right) d B_{s}^{j},\ t \in[0, T].
	\end{equation}\par
    We need show the following three lemmas before we present our proof of Theorem \ref{solution}.
	\begin{lem}\label{defined}
		Assume that (H1) and (H2) hold. Then, the mapping $U(F)={\mathbb{F}}^{{X^F}}:{\cal {D}}^{T}\to{\cal {D}}^{T}$ is well-defined.
	\end{lem}
	\begin{proof} Fixing $F \in {\cal {D}}^{T}$, the equation (\ref{GSDEs}) then becomes a SDE driven by $G$-Brownian motion. Thanks to the assumptions (H1) and (H2), it can be deduced from Theorem 5.1.3 in \cite{peng2} that there is a unique solution $\left(X_{t}^{F}\right)_{t \in[0, T]} \in \bar M_G^2\left(0, T ; {\mathbb{R}^n}\right)$ to the equation (\ref{GSDEs}). We then only need to prove that $U(F) \in {\cal {D}}^{T}$. In fact, it follows from the sub-additivity of $\hat {\mathbb{E}}$ that
	\[\begin{aligned}&
\mathop {\sup }\limits_{t \in [0,T]} \mathop {\sup }\limits_{{C_\varphi } \le 1} \left| U(F)(\varphi)- \varphi \left( \bf{0} \right)\right|\\&
		=\mathop {\sup }\limits_{t \in [0,T]} \mathop {\sup }\limits_{{C_\varphi } \le 1} \left| {\hat {\mathbb{E}}\left[ {\varphi \left( {X_t^F} \right)} \right] - \varphi \left( \bf{0} \right)} \right|\\&
		\le \mathop {\sup }\limits_{t \in [0,T]} \mathop {\sup }\limits_{{C_\varphi } \le 1} \hat {\mathbb{E}}\left[ {\left| {\varphi \left( {X_t^F} \right) - \varphi \left( \bf{0} \right)} \right|} \right]\\&
		\le \mathop {\sup }\limits_{t \in [0,T]} \hat {\mathbb{E}}\left[ {\left| {X_t^F} \right|} \right]\\&
        < \infty ,
	\end{aligned}\]
where the last inequality is due to the estimate for the solution of $G$-SDE (see Proposition 3.3 in Lin \cite{lin1}). Thus, the mapping $U:{\cal {D}}^{T}\to{\cal {D}}^{T}$ is well-defined. The proof of Lemma \ref{defined} is completed.
\end{proof}
	\begin{lem}\label{uniqueness}
		For $0 \le t \le T$ and ${F^1},{F^2} \in {\cal {D}}^{T}$, we have
		\[{\left[ {d_1^t\left( {U\left( {{F^1}} \right),U\left( {{F^2}} \right)} \right)} \right]^2} \le {C_T}{\int_0^t {\left[ {d_1^s\left( {{F^1},{F^2}} \right)} \right]} ^2}ds,\]
		where $C_T$ is a constant depending only on the constants $K$ and $T$.
	\end{lem}
	\begin{proof}
		For $F^{1}, F^{2} \in {\cal {D}}_{T}$, we notice that $U({F^1})$ and $U({F^2})$ are the distributions of ${X^1}$, ${X^2}$ respectively, where ${X^1}$ and ${X^2}$ are determined by the following
		\[X_{t}^{1}=x_0+\int_{0}^{t} b\left(s, X_{s}^{1}, F_{s}^{1}\right) d s+\int_{0}^{t} h_{i j}\left(s, X_{s}^{1}, F_{s}^{1}\right) d\langle B\rangle_{s}^{i j}+\int_{0}^{t} \sigma_{j}\left(s, X_{s}^{1}, F_{s}^{1}\right) d B_{s}^{j},\]
		\[X_{t}^{2}=x_0+\int_{0}^{t} b\left(s, X_{s}^{2}, F_{s}^{2}\right) d s+\int_{0}^{t} h_{i j}\left(s, X_{s}^{2}, F_{s}^{2}\right) d\langle B\rangle_{s}^{i j}+\int_{0}^{t} \sigma_{j}\left(s, X_{s}^{2}, F_{s}^{2}\right) d B_{s}^{j}\]
		for $t \in[0, T]$. Then by H\"{o}lder inequality of $\hat{\mathbb{E}}$ (see Proposition 1.4.2 in Peng \cite{peng2}), we get
		\[ \begin{aligned}
			{\left[ {d_1^t\left( {U\left( {{F^1}} \right),U\left( {{F^2}} \right)} \right)} \right]^2} &= \mathop {\sup }\limits_{s \in [0,t]} \mathop {\sup }\limits_{
				\scriptstyle{C_\varphi } \le {1}\hfill} {\left| {\hat{\mathbb{E}}\left[ {\varphi \left( {X_s^1} \right)} \right] - \hat{\mathbb{E}}\left[ {\varphi \left( {X_s^2} \right)} \right]} \right|^2} \\
			&  \le \mathop {\sup }\limits_{s \in [0,t]} \mathop {\sup }\limits_{
				\scriptstyle{C_\varphi } \le {1}\hfill} \hat{\mathbb{E}}\left[ {{{\left| {\varphi \left( {X_s^1} \right) - \varphi \left( {X_s^2} \right)} \right|}^2}} \right] \\&
			\le  \mathop {\sup }\limits_{s \in [0,t]} \hat{\mathbb{E}}\left[ {{{\left| {X_s^1 - X_s^2} \right|}^2}} \right].
		\end{aligned}
		\]
	Meanwhile, for $0\le t\le T'\leq T$, it follows from Lemma \ref{inequality1} and Lemma \ref{inequality3} that
		\[\begin{aligned}&
			\mathop {\sup }\limits_{s \in [0,t]} \hat{\mathbb{E}}\left[ {{{\left| {X_s^1 - X_s^2} \right|}^2}} \right]\\&
			\le \mathop {\sup }\limits_{s \in [0,t]} C\hat{\mathbb{E}}\left[ {{{\left| {\int_0^s {\left( {b\left( {u,X_u^1,F_u^1} \right) - b\left( {u,X_u^2,F_u^2} \right)} \right)} du} \right|}^2}} \right]\\&
			\quad+ \mathop {\sup }\limits_{s \in [0,t]} C\hat{\mathbb{E}}\left[ {{{\left| {\int_0^s {\left( {{h_{ij}}\left( {u,X_u^1,F_u^1} \right) - {h_{ij}}\left( {u,X_u^2,F_u^2} \right)} \right)} d\left\langle B \right\rangle _u^{ij}} \right|}^2}} \right]\\&
			\quad+ \mathop {\sup }\limits_{s \in [0,t]} C\hat{\mathbb{E}}\left[ {{{\left| {\int_0^s {\left( {{\sigma _j}\left( {u,X_u^1,F_u^1} \right) - {\sigma _j}\left( {u,X_u^2,F_u^2} \right)} \right)} dB_u^j} \right|}^2}} \right]\\&
			\le C\int_0^t {\hat{\mathbb{E}}\left[ {{{\left| {b\left( {s,X_s^1,F_s^1} \right) - b\left( {s,X_s^2,F_s^2} \right)} \right|}^2}} \right]} ds\\&
			\quad + C\int_0^t {\hat{\mathbb{E}}\left[ {{{\left| {{h_{ij}}\left( {s,X_s^1,F_s^1} \right) - {h_{ij}}\left( {s,X_s^2,F_s^2} \right)} \right|}^2}} \right]} ds\\&
			\quad + C\int_0^t {\hat{\mathbb{E}}\left[ {{{\left| {{\sigma _j}\left( {s,X_s^1,F_s^1} \right) - {\sigma _j}\left( {s,X_s^2,F_s^2} \right)} \right|}^2}} \right]} ds\\&
			\le C\int_0^t {\hat{\mathbb{E}}\left[ {{{\left| {X_s^1 - X_s^2} \right|}^2}} \right]} ds + C\int_0^t {{{\left[ {d_1(F_s^1,F_s^2)} \right]}^2}} ds\\&
			\le C\int_0^t {\mathop {\sup }\limits_{s' \in [0,s]} } \hat{\mathbb{E}}\left[ {{{\left| {X_{s'}^1 - X_{s'}^2} \right|}^2}} \right]ds + C{\int_0^{T'} {\left[ {{d_1^s}({F^1},{F^2})} \right]} ^2}ds.
		\end{aligned}\]
Next by Gronwall inequality, we have
		\[\mathop {\sup }\limits_{s \in [0,t]} \hat{\mathbb{E}}\left[ {{{\left| {X_s^1 - X_s^2} \right|}^2}} \right] \le C{e^{Ct}}\int_0^{T'} {{{\left[ {d_1^s({F^1},{F^2})} \right]}^2}} ds, \text{ for } 0\le t\le T'.\]
		Let $t=T'$, we obtain the following
		\[\begin{aligned}
			{\left[ {d_1^t\left( {U\left( {{F^1}} \right),U\left( {{F^2}} \right)} \right)} \right]^2} &\le \mathop {\sup }\limits_{s \in [0,t]} \hat{\mathbb{E}}\left[ {{{\left| {X_s^1 - X_s^2} \right|}^2}} \right]\\&
			\le C{e^{Ct}}{\int_0^t {\left[ {{d_1^s}({F^1},{F^2})} \right]} ^2}ds\\&
			= {C_T}{\int_0^t {\left[ {{d_1^s}({F^1},{F^2})} \right]} ^2}ds,
		\end{aligned}\]
where $C_T=C{e^{CT}}$, $C$ is a constant depending only on $K$ and $T$, and may vary line by line. The proof of Lemma \ref{uniqueness} is completed.
\end{proof}
	\begin{lem}\label{existence}
		For $k \ge 1$, $F\in {\cal {D}}^T$, $0\le t\le T$, we have
		\begin{equation}\label{iterate}
			{\left[ {d_1^t\left( {{U^{(k + 1)}}\left( F \right),{U^{(k)}}\left( F \right)} \right)} \right]^2} \le {C_T}^{k}\frac{{{t^{k}}}}{{k!}}{\left[ {d_1^t\left( {U\left( F \right),F} \right)} \right]^2},
		\end{equation}
		where $C_T$ is the same constant as in Lemma \ref{uniqueness}.
	\end{lem}
	\begin{proof}For any $F\in {\cal {D}}^{T}$, $k \ge 1$, let ${U^{(k)}}\left( F \right) = U\left( {{U^{(k - 1)}}\left( F \right)} \right)$ and
		${U^{(0)}}\left( F \right)= F$.
		Then ${{U^{(k )}}\left( F \right)}$ is the distribution of ${X^{(k )}}$, where $X^{(k)}$ for each $k\ge 1$ is defined by the following
		\[\begin{aligned}
			X_t^{(k)} &= x_0+ \int_0^t b \left( {s,X_s^{(k)},U_s^{(k - 1)}\left( F \right)} \right)ds + \int_0^t {{h_{ij}}} \left( {s,X_s^{(k)},U_s^{(k - 1)}\left( F \right)} \right)d\langle B\rangle _s^{ij}\\&
			\quad\ + \int_0^t {{\sigma _j}} \left( {s,X_s^{(k)},U_s^{(k - 1)}\left( F \right)} \right)dB_s^j,\quad t \in[0, T].
		\end{aligned}\]
By Lemma \ref{uniqueness}, we have
		\[\left[{d_1^t}\left( {{U^{(2)}}\left( F \right),U^{(1)}\left( F \right)} \right)\right]^2 \le {C_T}\int_{0}^{t}\left({d_1^s}\left( {U\left( F \right),F} \right)\right)^2 ds \le {C_T}t\left[{d_1^t}\left( {U\left( F \right),F}\right) \right]^2, \ 0\le t\le T.\]
		We next use mathematical induction argument to show the inequality (\ref{iterate}). Suppose the inequality (\ref{iterate}) holds for $k-1$, namely
		\[{\left[ {d_1^t\left( {{U^{(k)}}\left( F \right),{U^{(k - 1)}}\left( F \right)} \right)} \right]^2} \le {C_T}^{k-1}\frac{{{t^{k-1}}}}{{{(k-1)}!}}{\left[ {d_1^t\left( {U\left( F \right),F} \right)} \right]^2}, \ 0\le t\le T.\]
Similar to the proof of Lemma \ref{uniqueness}, we have
		\[\begin{aligned}&
			{\left[ {d_1^t\left( {{U^{(k + 1)}}\left( F \right),{U^{(k)}}\left( F \right)} \right)} \right]^2}\\&
			= \mathop {\sup }\limits_{s \in [0,t]} \mathop {\sup }\limits_{
				\scriptstyle{C_\varphi } \le {1}\hfill} {{\left| {\hat{\mathbb{E}}\left[ {\varphi \left( {X_s^{(k+1)}} \right)} \right] - \hat{\mathbb{E}}\left[ {\varphi \left( {X_s^{(k )}} \right)} \right]} \right|}^2}\\&
			\le \mathop {\sup }\limits_{s \in [0,t]} \hat{\mathbb{E}}\left[ {{{\left| {X_s^{(k + 1)} - X_s^{(k)}} \right|}^2}} \right]\\&
			\le {C_T}\int_0^t {{{\left[ {d_1^s\left( {{U^{(k)}}\left( F \right),{U^{(k - 1)}}\left( F \right)} \right)} \right]}^2}} ds\\&
            \le {C_T}\int_0^t {C_T^{k{\rm{ - }}1}\frac{{{s^{k{\rm{ - }}1}}}}{{(k - 1)!}}{{\left[ {d_1^s(U\left( F \right),F)} \right]}^2}} ds\\&
            \le C_T^k\frac{{{t^k}}}{{k!}}{\left[ {d_1^t\left( {U\left( F \right),F} \right)} \right]^2}.
		\end{aligned}\]
The proof of Lemma \ref{existence} is completed.
\end{proof}
       \noindent {\bf \it Proof of Theorem \ref{solution}}.
		First, we notice that if there exist $ {F^1},{F^2} \in {\cal {D}}^T$  such that $U({F^1}) = {F^1},U({F^2}) = {F^2}$, then it follows from Lemma \ref{uniqueness} that for $ 0 \le t \le T$,
		\[{\left[ {d_1^t\left( {{F^1},{F^2}} \right)} \right]^2} = {\left[ {d_1^t\left( {U\left( {{F^1}} \right),U\left( {{F^2}} \right)} \right)} \right]^2} \le {C_T}\int_0^t {{{\left[ {d_1^s\left( {{F^1},{F^2}} \right)} \right]}^2}} ds.\]
		Using Gronwall inequality, we have
		\begin{center}
			${d_1^t}\left( {{F^1},{F^2}} \right) =0$, for $0 \le t \le T.$
		\end{center}
		Therefore, we conclude that the mapping $U:{\cal {D}}^{T}\to{\cal {D}}^{T}$ has at most one fixed point.\par
		Second, it follows from Proposition 3.3 of Lin \cite{lin1} that $\mathop {\sup }\limits_{t \in [0,T]} \hat{\mathbb{E}}\left[ {\left| {X_t^{(1)}} \right|} \right] < \infty $, then due to the fact that $F \in {\cal {D}}^{T}$, we have
    	\[\begin{aligned}
		d_1^T\left( {U\left( F \right),F} \right) &\le d_1^T\left( {U\left( F \right),{\mathbb{F}_{\bf{0}}}} \right) + d_1^T\left( {F,{\mathbb{F}_{\bf{0}}}} \right)\\&
		= \mathop {\sup }\limits_{t \in [0,T]} \mathop {\sup }\limits_{{C_\varphi } \le 1} \left| {\hat{\mathbb{E}}\left[ {\varphi \left( {X_t^{(1)}} \right)} \right] - \varphi \left( {\bf{0}} \right)} \right| + d_1^T\left( {F,{\mathbb{F}_{\bf{0}}}} \right)\\&
		\le \mathop {\sup }\limits_{t \in [0,T]} \hat{\mathbb{E}}\left[ {\left| {X_t^{(1)}} \right|} \right] + d_1^T\left( {F,{\mathbb{F}_{\bf{0}}}} \right)\\&
		< \infty .
	    \end{aligned}\]
        From Lemma \ref{existence}, one gets that for $m > n \ge 0$
        \[\begin{aligned}
        	d_1^T\left( {{U^{(m)}}\left( F \right),{U^{(n)}}\left( F \right)} \right) &\le \sum\limits_{k = n}^{m - 1} {d_1^T\left( {{U^{(k + 1)}}\left( F \right),{U^{(k)}}\left( F \right)} \right)} \\&
        	\le \sum\limits_{k = n}^\infty  {{{\left( {\frac{{{{\left( {{C_T}T} \right)}^k}}}{{k!}}} \right)}^{\frac{1}{2}}}d_1^T\left( {U\left( F \right),F} \right).}
        \end{aligned}\]
		Let $n \to \infty $, we have $d_1^T\left( {{U^{(m)}}\left( F \right),{U^{(n)}}\left( F \right)} \right) \to 0$.
		Thus ${\left\{ {{U^{(k)}}\left( F \right)} \right\}_{k=1}^{\infty}}$ is a Cauchy sequence in ${\cal {D}}^T$. Since ${ \cal {D}}^T$ is a complete metric space with respect to the metric $d_{r}^T$ (see our Proposition \ref{complete}), then ${\left\{ {{U^{(k)}}\left( F \right)} \right\}_{k=1}^{\infty}}$ converges. That is, there exists $ {F^ * } \in {\cal {D}}^T$ such that
		\[d_1^T\left( {{U^{(k)}}\left( F \right),{F^ * }} \right) \to 0, \text{ as }k \to \infty .\]  From Lemma \ref{uniqueness} and ${U^{(k)}}\left( F \right) = U\left( {{U^{(k - 1)}}\left( F \right)} \right)$, we have
		\[\begin{aligned}&
			d_1^T\left( {U\left( {{F^*}} \right),{F^*}} \right)\\&
			\le d_1^T\left( {{U^{(k)}}\left( F \right),{F^*}} \right) + d_1^T\left( {{U^{(k)}}\left( F \right),U\left( {{F^*}} \right)} \right)\\&
			\le d_1^T\left( {{U^{(k)}}\left( F \right),{F^*}} \right) + {C_T}Td_1^T\left( {{U^{(k - 1)}}\left( F \right),{F^*}} \right).
		\end{aligned}\]
		Let $k \to \infty $, we have $d_1^T\left( {U\left( {{F^*}} \right),{F^*}} \right) = 0$. Hence $F^*$ is the fixed point of $U$.

		If $\left(X_{t}\right)_{t \in[0, T]}$ is the solution of the distribution dependent $G$-SDE (\ref{DDGSDEs}), then the distribution of $\left(X_{t}\right)_{t \in[0, T]}$ is the fixed point of $U:{\cal {D}}^T\to{\cal {D}}^T$ and vice versa. Namely, the existence and uniqueness of the fixed point of $U:{\cal {D}}^T\to{\cal {D}}^T$ is equivalent to the existence and uniqueness of the solution $\left(X_{t}\right)_{t \in[0, T]} \in \bar M_G^2\left(0, T ; {\mathbb{R}^n}\right)$ of the distribution dependent $G$-SDE (\ref{DDGSDEs}). The proof of Theorem \ref{solution} is completed. $\hfill\Box$

	\begin{thm}\label{estimate1}
	For any $p \ge  2$, assume that (H1) and (H2) hold and $\phi(\cdot,{\bf{0}},\mathbb{F}_{\bf{0}})\in L^p(0,T;\mathbb{R}^n)$, that is there exists a positive constant $M$ such that
$ \int_{0}^T\left|\phi(s,{\bf{0}},\mathbb{F}_{\bf{0}})\right|^p ds\le M$, where $\phi  = b,{h_{ij}},{\sigma _j}$, respectively.
Then, we have the following estimate for the solution ${\left( {X_t^x} \right)_{t \in [0,T]}}$ of the distribution dependent $G$-SDE (\ref{DDGSDEs}) with initial condition $x_0=x$
    \[\hat{\mathbb{E}}\left[ {\mathop {\sup }\limits_{s \in [0,t]} {{\left| {X_s^x} \right|}^p}} \right] \le {C_1}{e^{{C_2}t}}, \ \ t \in [0,T],\]
    where
    \[\begin{aligned}&
    	{C_1} = {4^{p - 1}}{\left| x \right|^p} + {8^{p - 1}}M\left( {{T^{p - 1}} + {C_\sigma }{T^{p - 1}} + {C_p}\sigma _{{\mathbf{e_j}}{\mathbf{e_j}}^{T} }^{p}{T^{\frac{p}{2} - 1}}} \right),\\&
        {C_2} = {2^{4p - 3}}K^p\left( {{T^{p - 1}} + {C_\sigma }{T^{p - 1}} + {C_p}\sigma _{{\mathbf{e_j}}{\mathbf{e_j}}^{T} }^{p}{T^{\frac{p}{2} - 1}}} \right),\\&
        {{C_\sigma }}= {\left( {\frac{{\sigma _{({{\bf{e}}_i} + {{\bf{e}}_j}){{({{\bf{e}}_i} + {{\bf{e}}_j})}^T}}^2 + \sigma _{({{\bf{e}}_i} - {{\bf{e}}_j}){{({{\bf{e}}_i} - {{\bf{e}}_j})}^T}}^2}}{4}} \right)^p}
    \end{aligned}\]
    and $C_p$ is the constant as in Lemma \ref{inequality2}.
	\end{thm}
	\begin{proof}
For $p \ge 2$, we have
		\[\begin{aligned}
			{\left| {X_s^x} \right|^p} &\le {4^{p - 1}}\left( {{{\left| x \right|}^p} + {{\left| {\int_0^s b \left( {u,X_u^x,{\mathbb{F}}_u^{{X^x}}} \right)du} \right|}^p} + {{\left| {\int_0^s {{h_{ij}}} \left( {u,X_u^x,{\mathbb{F}}_u^{{X^x}}} \right)d\langle B\rangle _u^{ij}} \right|}^p}} \right.\\&
			\quad\quad\quad\quad\left. { + {{\left| {\int_0^s {{\sigma _j}} \left( {u,X_u^x,{\mathbb{F}}_u^{{X^x}}} \right)dB_u^j} \right|}^p}} \right).
		\end{aligned}\]
        It follows from H\"{o}lder inequality, the subadditivity of $G$-expectation  $\hat{\mathbb{E}}$ and Lemma \ref{inequality2}, Lemma \ref{inequality3} that
		\[\begin{aligned}&
			\hat{\mathbb{E}}\left[ {\mathop {\sup }\limits_{s \in [0,t]} {{\left| {X_s^x} \right|}^p}} \right] \\&\le {4^{p - 1}}\left( {{{\left| x \right|}^p} + \hat{\mathbb{E}}\left[ {\mathop {\sup }\limits_{s \in [0,t]} {{\left| {\int_0^s b \left( {u,{X_u^x},{\mathbb{F}}_u^{{X^x}}} \right)du} \right|}^p}} \right]} \right.\\&\quad
			\left. { + \hat{\mathbb{E}}\left[ {\mathop {\sup }\limits_{s \in [0,t]} {{\left| {\int_0^s {{h_{ij}}} \left( {u,{X_u^x},{\mathbb{F}}_u^{{X^x}}} \right)d\langle B\rangle _u^{ij}} \right|}^p}} \right] + \hat{\mathbb{E}}\left[ {\mathop {\sup }\limits_{s \in [0,t]} {{\left| {\int_0^t {{\sigma _j}} \left( {u,{X_u^x},{\mathbb{F}}_u^{{X^x}}} \right)dB_u^j} \right|}^p}} \right]} \right)\\&
			\le {4^{p - 1}}\left( |x|^p+{{t^{p - 1}}\int_0^t {\hat{\mathbb{E}}\left[ {{{\left| {b\left( {s,X_s^x,{\mathbb{F}}_s^{{X^x}}} \right)} \right|}^p}} \right]} ds + {C_\sigma }{t^{p - 1}}\int_0^t {\hat{\mathbb{E}}\left[ {{{\left| {{h_{ij}}\left( {s,X_s^x,{\mathbb{F}}_s^{{X^x}}} \right)} \right|}^p}} \right]} ds} \right.\\&
			\quad\left. { + {C_p}\sigma _{{\mathbf{e_j}}{\mathbf{e_j}}^{T} }^{p}{t^{\frac{p}{2} - 1}}\int_0^t {\hat{\mathbb{E}}\left[ {{{\left| {{\sigma _j}\left( {s,X_s^x,{\mathbb{F}}_s^{{X^x}}} \right)} \right|}^p}} \right]} ds} \right),
		\end{aligned}\]
where $ {{C_\sigma }}= {\left( {\frac{{\sigma _{({{\bf{e}}_i} + {{\bf{e}}_j}){{({{\bf{e}}_i} + {{\bf{e}}_j})}^T}}^2 + \sigma _{({{\bf{e}}_i} - {{\bf{e}}_j}){{({{\bf{e}}_i} - {{\bf{e}}_j})}^T}}^2}}{4}} \right)^p}$ and $C_p$ is the constant in Lemma \ref{inequality2}.
		From Assumption (H1), we get
		\[\begin{aligned}
			{\left| {\phi \left( {s,X_s^x,{\mathbb{F}}_s^{{X^x}}} \right)} \right|^p} &\le {2^{p - 1}}\left( {{{\left| {\phi \left( {s,X_s^x,{\mathbb{F}}_s^{{X^x}}} \right) - \phi \left( {s,{\bf{0}},\mathbb{F}_{\bf{0}}} \right)} \right|}^p} + {{\left| {\phi \left( {s,{\bf{0}},\mathbb{F}_{\bf{0}}} \right)} \right|}^p}} \right)\\&
			\le {2^{p - 1}}\left( { {K^p}{{\left( {\left| {X_s^x} \right| + {d_1}\left( {{\mathbb{F}}_s^{{X^x}},\mathbb{F}_{\bf{0}}} \right)} \right)}^p}}  + {{\left| {\phi \left( {s,{\bf{0}},\mathbb{F}_{\bf{0}}} \right)} \right|}^p}\right)\\&
			\le {2^{p - 1}}\left( { {K^p}{2^{p - 1}}{{\left| {X_s^x} \right|}^p} + {K^p}{2^{p - 1}}{{{{d_1}\left( {{\mathbb{F}}_s^{{X^x}},\mathbb{F}_{\bf{0}}} \right)} }^p}} +{{\left| {\phi \left( {s,{\bf{0}},\mathbb{F}_{\bf{0}}} \right)} \right|}^p}\right)\\&
	= {2^{p - 1}}\left( { {K^p}{2^{p - 1}}{{\left| {X_s^x} \right|}^p} + {K^p}{2^{p - 1}}{{ \mathop {\sup }\limits_{
			\scriptstyle{C_\varphi } \le {1}\hfill}\left|\hat{\mathbb{E}}[\varphi(X_s^{x})]-\varphi({\bf{0}})\right|^p }}} +{{\left| {\phi \left( {s,{\bf{0}},\mathbb{F}_{\bf{0}}} \right)} \right|}^p}\right)\\&
\le {2^{p - 1}}\left( { {K^p}{2^{p - 1}}{{\left| {X_s^x} \right|}^p} +  {K^p}{2^{p - 1}}\hat{\mathbb{E}}\left[ {{{\left| {X_s^x} \right|}^p}} \right]}+{{\left| {\phi \left( {s,{\bf{0}},\mathbb{F}_{\bf{0}}} \right)} \right|}^p} \right),
		\end{aligned}\]
then summarily we have
	\[\int_0^t {\hat{\mathbb{E}}\left[ {{{\left| {\phi \left( {s,X_s^x,F_s^{{X^x}}} \right)} \right|}^p}} \right]ds}  \le  {2^{p - 1}}M  +  {2^{2p - 1}}{K^p} \int_0^t {\hat{\mathbb{E}}\left[ {{{\left| {X_s^x} \right|}^p}}\right]ds},\]
	    for $\phi  = b,{h_{ij}},{\sigma _j}$, respectively. Thus,
	\[\begin{aligned}&
		\hat{\mathbb{E}}\left[ {\mathop {\sup }\limits_{s \in [0,t]} {{\left| {X_s^x} \right|}^p}} \right]\\ &\le {4^{p - 1}}\left[ {{\left| x \right|}^p} + \left( {{t^{p - 1}} + {C_\sigma }{t^{p - 1}} + {C_p}\sigma _{{\mathbf{e_j}}{\mathbf{e_j}}^{T} }^{p}{t^{\frac{p}{2} - 1}}} \right)\left( {{2^{p - 1}}M + {{2^{2p - 1}}{K^p} \int_0^t {\hat{\mathbb{E}}\left[ {{{\left| {X_s^x} \right|}^p}} \right]ds} }}\right) \right]\\&
		\le {4^{p - 1}}{\left| x \right|^p} + {8^{p - 1}}M\left( {{T^{p - 1}} + {C_\sigma }{T^{p - 1}} + {C_p}\sigma _{{\mathbf{e_j}}{\mathbf{e_j}}^{T} }^{p}{T^{\frac{p}{2} - 1}}} \right)\\&
     	\quad	+ {2^{4p - 3}}K^p\left( {{T^{p - 1}} + {C_\sigma }{T^{p - 1}} + {C_p}\sigma _{{\mathbf{e_j}}{\mathbf{e_j}}^{T} }^{p}{T^{\frac{p}{2} - 1}}} \right)\int_0^t {\hat{\mathbb{E}}\left[ {{{\left| {X_s^x} \right|}^p}} \right]ds}\\&
        \le {C_1} + {C_2}\int_0^t {\mathop {\sup }\limits_{s' \in [0,s]} \hat{\mathbb{E}}\left[ {{{\left| {X_{s'}^x} \right|}^p}} \right]ds}.
	\end{aligned}\]
		Using Gronwall inequality, we have
		$\hat{\mathbb{E}}\left[ {\mathop {\sup }\limits_{s \in [0,t]} {{\left| {X_s^x} \right|}^p}} \right] \le {C_1}{e^{{C_2}t}}.$
	\end{proof}
	\begin{thm}\label{estimate2}
		For any $p \ge 2$, assume that (H1) and (H2) hold, then there exists a positive constant $C_3$ such that for all $x,y \in {\mathbb{R}^n},~t \in \left[ {0,T} \right]$,
		\[\hat{\mathbb{E}}\left[ {\mathop {\sup }\limits_{s \in [0,t]} {{\left| {X_s^x - X_s^y} \right|}^p}} \right] \le C_3{\left| {x - y} \right|^p},\]
		where $C_3={4^{p - 1}}\exp \left\{ {{2^{3p - 2}}{K^p}\left( {{T^{p-1}} + {C_\sigma }{T^{p-1}} + {C_p}{T^{\frac{p}{2}-1}}} \right)t} \right\}.$
	\end{thm}
   \begin{proof}
   	For $p \ge 2$, we have
   \[\begin{aligned}
   	{\left| {X_s^x - X_s^y} \right|^p} &\le {4^{p - 1}}\left({{{\left| {x - y} \right|}^p} + {{\left| {\int_0^s {\left( {b\left( {u,X_u^x,{\mathbb{F}}_u^{{X^x}}} \right) - b\left( {u,X_u^y,{\mathbb{F}}_u^{{X^y}}} \right)} \right)} du} \right|}^p}} \right.\\& \quad
   	+ {\left| {\int_0^s {\left( {{h_{ij}}\left( {u,X_u^x,{\mathbb{F}}_u^{{X^x}}} \right) - {h_{ij}}\left( {u,X_u^y,{\mathbb{F}}_u^{{X^y}}} \right)} \right)} d\langle B\rangle _u^{ij}} \right|^p}\\& \quad
   	\left. { + {{\left| {\int_0^s {\left( {{\sigma _j}\left( {u,X_u^x,{\mathbb{F}}_u^{{X^x}}} \right) - {\sigma _j}\left( {u,X_u^y,{\mathbb{F}}_u^{{X^y}}} \right)} \right)} dB_u^j} \right|}^p}} \right).
   \end{aligned}\]
    It follows from H\"{o}lder inequality, the subadditivity of $G$-expectation  $\hat{\mathbb{E}}$ and Lemma \ref{inequality2}, Lemma \ref{inequality3} that
    \[\begin{aligned}
    	\hat{\mathbb{E}}\left[ {\mathop {\sup }\limits_{0 \le s \le t} {{\left| {X_s^x - X_s^y} \right|}^p}} \right] &\le {4^{p - 1}}\left( {{{\left| {x - y} \right|}^p} + {t^{p - 1}}\int_0^t {\hat{\mathbb{E}}\left[ {{{\left| {b\left( {s,X_s^x,{\mathbb{F}}_s^{{X^x}}} \right) - b\left( {s,X_s^y,{\mathbb{F}}_s^{{X^y}}} \right)} \right|}^p}} \right]} ds} \right.\\&\quad
    	+ {C_\sigma }{t^{p - 1}}\int_0^t {	\hat{\mathbb{E}}\left[ {{{\left| {{h_{ij}}\left( {s,X_s^x,{\mathbb{F}}_s^{{X^x}}} \right) - {h_{ij}}\left( {s,X_s^y,{\mathbb{F}}_s^{{X^y}}} \right)} \right|}^p}} \right]} ds\\&\quad
    	\left. { + {C_p}\sigma _{{\mathbf{e_j}}{\mathbf{e_j}}^{T} }^{p}{t^{\frac{p}{2} - 1}}\int_0^t {	\hat{\mathbb{E}}\left[ {{{\left| {{\sigma _j}\left( {s,X_s^x,{\mathbb{F}}_s^{{X^x}}} \right) - {\sigma _j}\left( {s,X_s^y,{\mathbb{F}}_s^{{X^y}}} \right)} \right|}^p}} \right]} ds} \right).
    \end{aligned}\]
By Assumption (H2), we get
   \[\begin{aligned}
   	{\left| {\phi \left( {s,X_s^x,{\mathbb{F}}_s^{{X^x}}} \right) - \phi \left( {s,X_s^y,{\mathbb{F}}_s^{{X^y}}} \right)} \right|^p} &\le {2^{p - 1}}{K^p}\left( {{{\left| {X_s^x - X_s^y} \right|}^p} + {{\left| {{d_1}\left( {{\mathbb{F}}_s^{{X^x}},{\mathbb{F}}_s^{{X^y}}} \right)} \right|}^p}} \right)\\&
   	\le {2^{p - 1}}{K^p}\left( {{{\left| {X_s^x - X_s^y} \right|}^p} + \hat{\mathbb{E}}\left[ {{{\left| {X_s^x - X_s^y} \right|}^p}} \right]} \right),
   \end{aligned}\]
    thus summarily
    \[\int_0^t {\hat{\mathbb{E}}\left[ {{{\left| {\phi \left( {s,X_s^x,{\mathbb{F}}_s^{{X^x}}} \right) - \phi \left( {s,X_s^y,{\mathbb{F}}_s^{{X^y}}} \right)} \right|}^p}} \right]ds}  \le {2^{p }}{K^p}\int_0^t {\hat{\mathbb{E}}\left[ {{{\left| {X_s^x - X_s^y} \right|}^p}} \right]ds} ,\]
   for $\phi  = b,{h_{ij}},{\sigma _j}$, respectively. Hence, we obtain
    \[\begin{aligned}&
    	\hat{\mathbb{E}}\left[ {\mathop {\sup }\limits_{0 \le s \le t} {{\left| {X_s^x - X_s^y} \right|}^p}} \right] \\&\le {4^{p - 1}}{\left| {x - y} \right|^p} + {4^{p - 1}}{2^{p }}{K^p}\left( {{t^{p - 1}} + {C_\sigma }{t^{p - 1}} + {C_p}\sigma _{{\mathbf{e_j}}{\mathbf{e_j}}^{T} }^{p}{t^{\frac{p}{2} - 1}}} \right)\int_0^t {\hat{\mathbb{E}}\left[ {{{\left| {X_s^x - X_s^y} \right|}^p}} \right]} ds\\&
        \le {4^{p - 1}}{\left| {x - y} \right|^p} + {C}\int_0^t {\hat{\mathbb{E}}\left[ {\mathop {\sup }\limits_{0 \le s' \le s} {{\left| {X_{s'}^x - X_{s'}^y} \right|}^p}} \right]} ds,
    \end{aligned}\]
    where ${C} ={2^{3p - 2}}{K^p}\left( {{T^{p - 1}} + {C_\sigma }{T^{p - 1}} + {C_p}\sigma _{{\mathbf{e_j}}{\mathbf{e_j}}^{T} }^{p}{T^{\frac{p}{2} - 1}}} \right).$
    Furthermore, by Gronwall inequality, we derive that
    \[\begin{aligned}
        \hat{\mathbb{E}}\left[ {\mathop {\sup }\limits_{0 \le s \le t} {{\left| {X_s^x - X_s^y} \right|}^p}} \right] &\le {4^{p - 1}}{\left| {x - y} \right|^p}{e^{{C}t}}\\& = C_3{\left| {x - y} \right|^p},
    \end{aligned}\]
    where $C_3={4^{p - 1}}{e^{{C}t}}$.
     The proof of Theorem \ref{estimate2} is thus completed.
   \end{proof}

\begin{example}
Let $b'(t,x,y), h'_{i j}(t,x,y), \sigma'_{j}(t,x,y): [0, T] \times {\mathbb{R}^n} \times {\mathbb{R}^n}\rightarrow {\mathbb{R}^n}$ satisfy the following  assumptions (A1) and (A2):\\
		(A1) there exists $K>0$ such that for all $t \in[0, T],\  x_1,x_2,y_1, y_2 \in \mathbb{R}^n$, we have
		 \[\begin{aligned}&
  \left| {b'(t,x_1,y_1) - b'(t,x_2,y_2)} \right| + \left| {{h'_{ij}}(t,x_1,y_1) - {h'_{ij}}(t,x_2,y_2)} \right| + \left| {{\sigma'_j}(t,x_1,y_1) - {\sigma' _j}(t,x_2,y_2)} \right| \\&\leq K(|x_1-x_2|+|y_1-y_2|);
\end{aligned}\]
(A2) for any fixed $(t,x)$, $b'(t,x,\cdot), h'_{i j}(t,x,\cdot), \sigma'_{j}(t,x,\cdot)$ are bounded functions on $\mathbb{R}^n$.

Specify $b,h_{i j},\sigma_{j}:[0, T] \times {\mathbb{R}^n} \times {\cal {D}} \rightarrow {\mathbb{R}^n}$, respectively,
by the following
$$b(t,x,f):=f(b'(t,x,\cdot)),\  h_{i j}(t,x,f):=f(h'_{i j}(t,x,\cdot)),\  \sigma_{j}(t,x,f):=f(\sigma'_{j}(t,x,\cdot)).$$
It can be checked that $b,h_{i j},\sigma_{j}$ fulfil Assumptions (H1) and (H2). Hence, the equation (\ref{DDGSDEs}) becomes the following
  \[ {X_t} = x_0 + \int_0^t \hat{\mathbb{E}} [b'(s,x,{X_s})]{|_{x = {X_s}}}ds + \int_0^t \hat{\mathbb{E}} [{h'_{ij}}(s,x,{X_s})]{|_{x = {X_s}}}d\langle B\rangle _s^{ij} + \int_0^t \hat{\mathbb{E}} [{\sigma' _j}(s,x,{X_s})]{|_{x = {X_s}}}dB_s^j,\]
which coincides with the mean-field $G$-SDE (\ref{sunsde}) introduced by Sun in \cite{sun2020}.
\end{example}

\vspace*{1em}
\noindent\textbf{Acknowledgement}

The work is supported by the National Key Research and Development Program of China (No. 2018YFA0703900) and the Natural Science Foundation of Shandong Province (No. ZR2021MA098 and ZR2019ZD41).


\begin{thebibliography}{99}
	\linespread{0}\addtolength{\itemsep}{-1.0ex}
		\bibitem{bai}
	    Bai, X., Lin, Y.: On the existence and uniqueness of solutions to stochastic differential equations driven by G-brownian motion with integral-lipschitz coefficients. {\it Acta Mathematicae Applicatae Sinica}, {\bf 30(3)}, 589-610 (2014)
    	\bibitem{lipeng}
        Buckdahn, R., Li, J., Peng, S., Rainer, C.: Mean-field stochastic differential equations and associated PDEs. {\it The Annals of Probability}, {\bf 45(2)}, 824-878 (2017)
		\bibitem{denis}
		Denis, L., Hu, M., Peng, S.: Function spaces and capacity related to a sublinear expectation: application to G-Brownian motion paths. {\it Potential analysis}, {\bf 34(2)}, 139-161 (2011)
		\bibitem{gao}
		Gao, F.: Pathwise properties and homeomorphic flows for stochastic differential equations driven by G-Brownian motion. {\it Stochastic Processes and their Applications}, {\bf 119(10)}, 3356-3382 (2009)
        \bibitem{hurenhe} Hu, L., Ren, Y., He, Q.: Pantograph stochastic differential equations driven by G-Brownian motion. {\it Journal of Mathematical Analysis and Applications}, {\bf 480(1)}, 123381 (2019)
        \bibitem{hujiliu} Hu, M., Ji, S.,  Liu, G.: On the strong Markov property for stochastic differential equations driven by G-Brownian motion. {\it Stochastic Processes and their Applications}, {\bf 131}, 417-453 (2021)
        \bibitem{hujisong} Hu, M., Ji, S., Song, Y.: Backward stochastic differential equations driven by G-Brownian motion. {\it Stochastic Processes and their Applications}, {\bf 124(1)},  759-784 (2014)
        \bibitem{hujiang} Hu, M., Jiang, L.: An efficient numerical method for forward-backward stochastic differential equations driven by G-Brownian motion. {\it Applied Numerical Mathematics}, {\bf 165}, 578–597 (2021)
     	\bibitem{huang}
	    Huang, X., Yang, F.: Distribution-dependent SDEs with H{\"o}lder continuous drift and $\alpha$-stable noise. {\it Numerical Algorithms}, {\bf 86(2)}, 813-831 (2021)
	    \bibitem{kac}
	    Kac, M.: Foundations of kinetic theory. {\it In Proceedings of The third Berkeley symposium on mathematical statistics and probability}, Vol. 3, 171-197 (1956)
	    \bibitem{Kan}
	    Kantorovich, L. V.,  Rubinshten, G. S.: On a space of completely additive functions. {\it Vestnik Leningrad Univ}, {\bf 13(7)}, 52-59 (1958)
		\bibitem{lin1}
		Lin, Q.: Differentiability of stochastic differential equations driven by the G-Brownian motion. {\it Science in China A: Mathematics}, {\bf 56(5)}, 1087-1107 (2013)
		\bibitem{lin2}
		Lin, Q.: Some properties of stochastic differential equations driven by the G-Brownian motion. {\it Acta Mathematica Sinica}, {\bf 56(5)}, English Series, {\bf 29(5)}, 923-942 (2013)
		\bibitem{mckean}
		McKean Jr, H. P.: A class of Markov processes associated with nonlinear parabolic equations. {\it Proceedings of the National Academy of Sciences}, {\bf 56(6)}, 1907-1911 (1966)
		\bibitem{mel}
		M\'{e}l\'{e}ard, S.: Asymptotic behaviour of some interacting particle systems; McKean-Vlasov and Boltzmann models.{\it Probabilistic models for nonlinear partial differential equations}, {\bf 1627}, 42-95 (1996)
		\bibitem{peng1}
		Peng, S.: G-expectation, G-Brownian motion and related stochastic calculus of It\^{o} type. {\it Stochastic analysis and applications}. Springer, Berlin, Heidelberg, 541-567 (2007)
		\bibitem{peng2}
		Peng, S.: Nonlinear expectations and stochastic calculus under uncertainty: with robust CLT and G-Brownian motion (Vol. 95). Springer Nature (2019)
		\bibitem{roc}
		R{\"o}ckner, M., Zhang, X.: Well-posedness of distribution dependent SDEs with singular drifts. {\it Bernoulli}, {\bf 27(2)}, 1131-1158 (2021)
        \bibitem{sun2020}		
        Sun, S.: Mean-field backward stochastic differential equations driven by $G$-Brownian motion and related partial differential equations. {\it Mathematical Methods in the Applied Sciences}, {\bf 43}, 7484–7505 (2020)
        \bibitem{sun2021}	
        Sun, S.: Mean-field backward stochastic differential equations driven by $G$-Brownian motion with uniformly continuous coefficients. {\it Applicable Analysis}, 1-23 (2021)
        \bibitem{szn}
		Sznitman, A. S.: Topics in propagation of chaos. Ecole d'Et\'{e} de probabilit\'{e}s de Saint-Flour XIX—1989. {\it Lecture Notes in Mathematics}, {\bf 1464}, 165-251 (1991)
        \bibitem{vlasov}		
        Vlasov, A. A.: The vibrational properties of an electron gas. {\it Soviet Physics Uspekhi}, {\bf 10(6)}, 721-733 (1968)
        \bibitem{wang}
		Wang, F.: Distribution dependent SDEs for Landau type equations. {\it Stochastic Processes and their Applications}, {\bf 128(2)}, 595-621 (2018)
        \bibitem{xugehu}
        Xu, L., Ge, S. S., Hu, H.: Boundedness and stability analysis for impulsive stochastic differential equations driven by G-Brownian motion. {\it International Journal of Control}, {\bf 92(3)}, 642-652 (2019)
	\end{thebibliography}
\end{document}